\documentclass[12pt]{article}

\usepackage[english]{babel}
\usepackage[margin=30mm,paperwidth=216mm, paperheight=303mm]{geometry}
\usepackage{fullpage}
\usepackage{parskip}
\usepackage{graphicx}
\usepackage{amssymb}
\usepackage{amsmath}
\usepackage{amsthm}
\usepackage{mathrsfs}
\usepackage{stmaryrd}
\usepackage{thmtools}
\usepackage{hyperref}
\usepackage[shortlabels]{enumitem}
\usepackage[backend = bibtex, maxnames=5]{biblatex}
\usepackage{tikz}
\usepackage{tikz-cd}
\tikzcdset{every label/.append style = {font = \normalsize}}
\usepackage{float}
\usepackage{diagbox}
\usepackage{subfig}
\usepackage[export]{adjustbox}
\usepackage{bm}

\newcommand{\Z}{\mathbb{Z}}

\newcommand{\N}{\mathbb{N}}

\newtheorem{tx}{Theorem}[section]
\newtheorem{tc}{Corollary}[tx]
\newtheorem{tl}[tx]{Lemma}
\newtheorem{tp}[tx]{Proposition}
\newtheorem{tcj}[tx]{Conjecture}
\theoremstyle{definition}
\newtheorem{td}[tx]{Definition}
\newtheorem*{te}{Example}

\newtheorem*{tr}{Remark}

\bibliography{ref}

\newcommand{\Change}[1]{\color{black}#1\color{black}}
\usepackage{authblk}

\begin{document}
\title{On Freiman's Theorem in a function field setting}
\author{Mieke Wessel}
\date{}

	\maketitle
\begin{abstract}
\noindent We prove some new instances of a conjecture of Bachoc, Couvreur and Z\'emor that generalizes Freiman's $3k-4$ Theorem to a multiplicative version in a function field setting. As a consequence we find that if $F$ is a rational function field over an algebraically closed field $K$ and $S \subset F$ a finite dimensional $K$-vector space such that $\dim S^2 = 2\dim S + 1$, then the conjecture holds.
\end{abstract}
\section{Introduction}
One of the goals of additive combinatorics is to show that a set $A \subset \Z$ with some form of additive structure has extra structure. For example, $A$ can contain many $3$-arithmetic progressions, or most elements in $A + A := \{a_1 + a_2 \mid a_1, a_2 \in A\}$ have many distinct pairs $(a_1, a_2) \in A^2$ representing them. This idea of linking discrete structures to each other is not unique to additive combinatorics and in particular we want to find analogues of theorems in additive combinatorics in the following function field setting.

Let $F$ be a function field over some algebraically closed field of constants $K$ and let $S$ be a finite dimensional $K$-vector space that is contained in $F$. If we require $S$ to have some \emph{multiplicative structure}, what does this say about other (structural) properties of $S$? 

\Change{Our current focus will be on generalizing }so-called inverse problems in additive combinatorics, which try to find the structure of some finite set $A \subset \Z$ when $|A+A|$ is small. \Change{A solution to this problem } is given by Freiman's Theorem, which gives bounds for $|A+A|$ such that $A$ is always contained in an $r$-dimensional arithmetic progression of bounded size. The relatively simple case $r = 1$ is known as Freiman's $3k-4$ theorem. 
\begin{tx}[Freiman's $3k-4$ Theorem \cite{freiman1973foundations}] \label{T3k4}
	Let $A \subset \Z$ such that $|A+A| \leq 3|A| - 4$. Then $A$ is contained in an arithmetic progression of length at most $|A+A| - |A| + 1$. 
\end{tx}

We are interested in the structure that $S$ can have when the dimension of $S^2 := \{s_1s_2 \mid s_1, s_2 \in S\}$, the set of products of elements from $S$, is small. \Change{The meaning of small is expressed by the combinatorial genus of $S$.}
\begin{td}
	The \textit{combinatorial genus} of $S$ is defined as
	$$\gamma := \dim S^2 - 2\dim S + 1.$$
\end{td}

In \cite{bachoc2018funcfieldfreiman} Bachoc, Couvreur and Z\'emour state a conjecture about how to generalize Freiman's $3k-4$ Theorem to the function field setting.

\begin{tcj} \label{CJfff}
	Let $K$ ben an algebraically closed field and let $F$ be an extension field of $K$. Let $S$ be a $K$-subspace of finite dimension in $F$ such that $K \subset S$. Suppose the combinatorial genus $\gamma$ of $S$ is smaller or equal than $\dim S - 3$. Then the genus $g$ of the field $K(S)$ satisfies $g \leq \gamma$ and there exists a Riemann-Roch space $L(D)$ that contains $S$ and such that $\dim L(D) \leq \dim S + \gamma - g$. 
\end{tcj}

\begin{tr}
	This is indeed a generalization of Theorem \ref{T3k4}, take $A \subset \Z$ a finite set and define $S := \{x^{a} \mid a \in A\}$ for $x$ some element in $F \backslash K$. 
\end{tr}

In the same paper this conjecture was proven for $\gamma$ equal to $0$ or $1$ and when additionally $g=0$, a possible explicit basis for $S$ was given.

	\begin{tx}[\cite{bachoc2018funcfieldfreiman}, Theorem 5.3 and Theorem 8.1]\label{Tknownbasis} Write $n$ for the dimension of $S$ over $K$ and assume $F = K(S)$ has genus $0$.
	\begin{enumerate}[\normalfont(i)]
		\item 	Suppose $n \geq 3$, and $\gamma  = 0$, then we can find $x \in F$ such that $F = K(x)$ and a basis for $S$ is given by $$1, x, x^2, \ldots, x^n. $$
		\item Suppose $n \geq 4$ and $\gamma = 1$, then we can find $x \in F$ and $\alpha \in K$ such that $F = K(x)$ and the basis of $S$ is given by either 
		\begin{align*}
			&1, x, x^2, \ldots, x^{n-2}, (x+\alpha)x^{n-1} \textup{ ~~or,} \\
			&1, (x+\alpha)x, (x+\alpha)x^2, \ldots, (x+\alpha)x^{n-1}. 
		\end{align*}
	\end{enumerate}
\end{tx}

The goal of this article is to show that the conjecture also holds when $g = 0$ and $\gamma = 2$. To prove this we use a similar setup as in \cite{bachoc2018funcfieldfreiman}, where they look at the natural filtration of $S$ with respect to some valuation $v$ of $F$. In [\cite{bachoc2018funcfieldfreiman}, Theorem 3.1] they first show that $|v(S)| = \dim S$ and then that one can always construct a `filtered basis' of $S$ with respect to $v$. 
\begin{td}
	Let $v$ be a valuation of $F$. A \textit{filtered basis} of $S$ (with respect to $v$) is a basis $(e_1, e_2, \ldots, e_n)$ such that 
	$$v(e_1) > v(e_2) > \ldots > v(e_n).$$
	A \textit{natural filtration} (with respect to $v$) is the sequence of subspaces $S_1, S_2, \ldots, S_n$ such that $S_i := \langle e_1, \ldots, e_i \rangle.$  This sequence is unique for each valuation and thus independent of the chosen filtered basis.
\end{td}
\begin{tr}
	By multiplying $S$ with $e_1^{-1}$ we may assume $e_1 = 1$ and thus $v(e_1) = 0$. 
\end{tr}
After fixing some valuation $v$ we may now define $\gamma_i := \dim S_i^2 - 2\dim S_i + 1$, the combinatorial genii of the $S_i$. As we will see in Lemma \ref{Lgrowinggenus} the finite sequence $(\gamma_i)_{1\leq i\leq n}$ is non-decreasing in $i$. This makes it possible to use Theorem \ref{Tknownbasis} to say something about the basis of $S_i$ when $\gamma_i$ equals either $0$ or $1$ and $i$ is not too small. These are the most important ideas used to prove the following main results of this paper.
%, which implies that Conjecture \ref{CJfff} holds when both $\gamma = 2$ and $g = 0$. 

\begin{tx}\label{Tmain}
	Let $S$ be an $n$-dimensional $K$-subspace with combinatorial genus $\gamma \leq n-3$ and suppose that the genus of $F = K(S)$ equals $0$. Write $S_1 \subset S_2 \subset \ldots \subset S_n$ for the natural filtration of $S$ with respect to some fixed valuation $v$ of $F$ and write $\gamma_i$ for the combinatorial genus of $S$. Suppose $\{\gamma_i \mid 1 \leq i \leq n\}$ equals either  $\{0, \gamma\}$ or $\{0, 1, \gamma\}$. Then, $S$ is contained in a space $L(D)$ with $\dim(L(D)) \leq n + \gamma$.
\end{tx}

It is immediately clear that this implies the following corollary. 
%
%The goal of this article is to prove Conjecture \ref{CJfff} for the case $g = 0$ and $\gamma = 2$, using that Theorem \ref{Tknownbasis} gives an explicit basis for $S_i$ when $\gamma_i := \dim S_i^2 - 2\dim S_i + 1 \in \{0, 1\}$ and $\gamma_i+ 3 \leq t \leq n-1$. This gives the following result.
\begin{tc}\label{Tgamma2}
		Let $S$ be a $K$-subspace of dimension $n \geq 5$ and combinatorial genus $\gamma = 2$ such that the genus $g$ of $F = K(S)$ equals $0$. Then $S$ is contained in a space $L(D)$ with $\deg(D) \leq n + 1$. 
\end{tc}

In Section \ref{Ssetup} the steps and tools for the proof of Theorem \ref{Tmain} are introduced and the cases where $\{\gamma_i \mid 1 \leq i \leq n \}$ equals $\{0, \gamma\}$ and $\{0, 1, \gamma\}$ are proven in Sections \ref{Sgam0} and \ref{Sgam01} respectively.

Finally, in Section \ref{Supper} we will consider another sequence of $\gamma_i$ and prove the following theorem.

\begin{tx}\label{Tmaxdiv}
	Let $S$, $F$ and $(\gamma_i)_{1\leq i\leq n}$ be defined as in Theorem \ref{Tmain}. Define $D_i$ as the divisor of least degree such that $S_i \subset L(D_i)$ and $\Delta_{\textup{Max}} := \max_{2\leq i \leq n}(\deg(D_i) - \deg(D_{i-1}))$. Suppose that $\gamma_3 = 0 $ and $\gamma_{\Delta_{\textup{Max}}} = 0$, then $\dim(L(D)) \leq n + \gamma$. 
\end{tx}

\subsection*{Notation and conventions}
From now on we will assume that $F = K(S)$ and that $F$ has genus $0$. This implies that $F = K(x)$ for some $x \in F$ and that all places of $F$ have the form $(x-\alpha)$ for $\alpha \in K$ except for $\frac 1x$, the place at infinity. We will also denote these places by $P_\alpha$ and $P_\infty$. The places correspond one to one to the valuations of $F$ which we write as $v_\alpha$ and $v_\infty$, respectively. 

With $1 = e_1, \ldots, e_n$ we will always denote a filtrated basis of $S$ with respect to $v_\infty$ and $-v_\infty(e_i)$ will also be called the degree of $e_i$. 

We write $\mathfrak p_i$ with $i$ an integer for the subset of $K[x]$ consisting of the polynomials in $x$ of degree at most $i$.

\section{Set-up}\label{Ssetup}
To prove Theorem \ref{Tmain} we will distinguish between quite a few cases. In this section we will outline the steps and some general tools used in the proof. 

The most important case distinction that we will make is whether $\{\gamma_1, \gamma_2, \ldots, \gamma_n \}$ equals $\{0, \gamma\}$ or  $\{0, 1, \gamma\}$. The following lemma confirms that for $\gamma = 2$, these are the only possibilities and that the genus of $K(S_i)$ will equal $0$ for all $S_i$ in the natural filtration of $S$.

\begin{tl} \label{Lgrowinggenus}
	Let $1\leq i \leq j \leq n$ and $g_i$ the genus of $K(S_i)$ then, $g_i = 0$ and $\gamma_i \leq \gamma_j$. 
\end{tl}
\begin{proof}
	The first statement follows directly from L\"uroth's Theorem\Change{[\cite{stichtenoth2009algfuncfields}, Proposition 3.5.9]}, since $g_n = g = 0$. 
	
	It suffices to prove the second statement for $j = i +1$. Note that, by the ultrametric property, $v(e_ie_{i-1})$ and $v(e_i^2)$ are not in $v(S_{i-1}^2)$. Hence, $\dim(S_i^2) \geq \dim(S_{i-1}^2) + 2$ and therefore,
	\[\gamma_i = \dim S_i^2 - 2\dim S_i + 1 \geq \dim S_{i-1}^2 + 2 - 2i + 1 = \gamma_{i-1}. \qedhere\]
\end{proof}
%https://mathoverflow.net/questions/151542/variety-with-perfect-function-field, Remark: when the characteristic of $K$ is 0 we can also proof that $g_i \leq g_j$ using the Hurwitz genus formula as written in Stichtentoch book.

We define $t$ as the integer between $2$ and $n$ such that $\gamma = \gamma_t > \gamma_{t-1}$. In the case that $\{\gamma_1, \ldots, \gamma_n \} = \{0, 1, \gamma\}$ we also define $t_1$ such that $1 = \gamma_{t_1} > \gamma_{t_1 - 1} = 0$. Note that the dimension of $S_{t}^2$ can be at most $t$ bigger than the dimension of $S_{t-1}^2$, because $S_t^2 = \langle S_{t-1}^2, e_1e_t, \ldots, e_{t-1}e_t, e_t^2 \rangle$ and therefore $\gamma_{t} \leq \gamma_{t-1} + t - 2$. This shows that it always holds that $t \geq \gamma - \gamma_{t-1} + 2$. 

\Change{
The proof when $t < \gamma + 3$ is slightly more involved than when $t \geq \gamma + 3$, but the general structure will be the same. We define $r := \max(t, \gamma +3)$, where for a first read-through one may assume $r = t$. }After fixing the set of values that the $\gamma_i$ can take and the values for $t$ and $t_1$, we will continue with the following steps. 

%\Comm{For the remainder of Section \ref{Ssetup} we will assume $t \geq \gamma + 3$, to keep things relatively uncluttered and such that we are always able to use Theorem \ref{Tknownbasis} for $S_{t-1}$. However, all of the upcoming lemmas and propositions will also apply for smaller $t$, unless explicitly mentioned. The main difference in the next steps for smaller $t$ is that we should replace `$t$' by `$\gamma + 3$' in certain places.}

%\Comm{Eigenlijk is dit vooral een recept voor $t \geq \gamma_{t-1} + 4$ of $t \geq \gamma + 3$, anders moet je namelijk naar $S_{\gamma + 3}$ gaan kijken en klopt dus stap 1 al niet, ik zou dit splitsen.}

	Steps:
\begin{enumerate}[1.]
	\item Determine the possible degrees of $e_1, e_2, \ldots, e_r$. 
	
	\item Look at the degrees of $\{e_ie_j \mid 1\leq i \leq j \leq r\}$ and try to determine which elements $s_1 \ldots, s_k$ are `needed' from $e_tS_r\cup e_{t+1}S_r\cup \ldots \cup e_rS_r$ such that $S_r^2 = \langle S_{t-1}^2, s_1, \ldots, s_k \rangle$. The ones not `needed', imply a relation between the $e_ie_j$. 
	
	\item Use the found relations to check if elements can and/or must have poles outside of infinity. 
	
	\item By similar techniques as for $e_r$, figure out the possible values of $n$ and poles of $e_{r+1}, \ldots, e_n$. This will lead to the smallest divisor $D$ such that $S \subset L(D)$. 
\end{enumerate}

%For now we will assume that $t \geq \gamma_{t-1} + 4$, such that we are always able to use Theorem \ref{Tknownbasis} on $S_{t-1}$. When this is not the case we can show $t < \gamma + 3$ and we will have to look at $S_{\gamma + 3}$ to find relations between the $S_i$, making the process slightly more complicated, even though the general idea stays the same and all the lemmas in the rest of this section will also hold for these cases. 

\subsection{Tools for Step 1}
Note that Freiman's $3k-4$ Theorem can be applied to the set $\deg(S)$ of degrees of $S$, because $\deg(S) + \deg(S) \subset \deg(S^2)$ and $|\deg(S^2)| = \dim(S^2)$. This means that this set must be an arithmetic progression missing at most $\gamma$ elements. The same then holds for $\deg(S_t)$, because it is a subset. If we know the greatest common divisor of $\deg(S_t)$, this idea will give us an upper bound for the degree of $e_t$. 

\Change{When $t \geq \gamma_{t-1} + 4$, the most important tool to find the degrees of $e_1, \ldots, e_{t-1}$ } is Theorem \ref{Tknownbasis}. However, a priori we only know the degrees in $K(S_{t-1}) \subset F$. The next lemma, which is a straightforward generalization of Lemma 7.11 in \cite{bachoc2018funcfieldfreiman}, can help us check whether $F = K(S_{t-1})$.

\begin{tl} \label{Lboundedfieldextension}
	Define $\delta := \max_{1\leq i \leq n-1}(\gamma_{i+1} - \gamma_i)$. Then $F = K(S) = K(S_{\delta+2})$. 
\end{tl}
\begin{proof}
	The codimension of $S_i^2$ in $S_iS_{i+1}$ equals $\gamma_{i+1} - \gamma_i + 1 \leq \delta + 1$. Since for $i \geq \delta + 2$ we know that $\dim S_i \geq \delta + 2$ the intersection $S_{i}^2\cap S_ie_{i+1}$ is non-empty. Therefore $e_{i+1}$ must be in $K(S_i)$ and $F = K(S_n) = \ldots = K(S_{\delta + 2})$. 
\end{proof}
\begin{tc}
	For any $t \in \N$ it holds that $F = K(S_t)$.
\end{tc}
\begin{proof}
	By Lemma \ref{Lboundedfieldextension} it suffices to prove that $t \geq \delta + 2$.  Let $j$ be such that $\delta = \gamma_j - \gamma_{j-1} = \dim S_j^2 - \dim S_{j-1}^2 - 2$. Note that $S_j^2  = \langle S_{j-1}^2, e_1e_j, e_2e_j, \ldots, e_j^2 \rangle$ and thus, 
	$$\delta = \dim S_j^2 - \dim S_{j-1}^2 - 2 \leq j - 2 \leq t - 2,$$
where the second inequality holds because $t$ is the largest index for which $\gamma_t - \gamma_{t-1}$ is non-zero.
\end{proof}

When $F \neq K(S_{t-1})$, it might occur that $\gcd(\deg(S_{t-1})) \neq 1$. The following lemma together with the previous corollary show that without loss of generality we may assume that $\gcd(\deg(S_t)) = 1$.

\begin{tl}\label{Lgcd}
	Suppose $F = K(S_i)$, then for all but finitely many valuations $v_\alpha$ with $\alpha \in K\cup\{\infty\}$ it holds that there are $s_1, s_2 \in S_i$, depending on $\alpha$, such that $v_\alpha(s_1) - v_\alpha(s_2) = 1$. In particular this implies $\gcd(v_\alpha(S_{\delta+2})) = 1$.
\end{tl}
\begin{proof}
	Consider two elements $s_1$ and $s_2$ of $S_i$ that have different degree. By multiplying by a suitable rational function $h$ we may assume $f := hs_1$ and $g := hs_2$ to be coprime polynomials of different degrees. Let $\alpha$ be any element of $F$ such that it is neither a zero of $f$ nor of $g$, then we can find a unique $a_\alpha\in K$ such that $(f + a_\alpha g)(\alpha) = 0$. Note that $a_\alpha \neq 0$ and $v_\alpha(f) = 0$. We claim that for all but finitely many $\alpha$ it also holds that $v_\alpha(f + a_\alpha g) = 1$. This would prove that the elements $s_1$ and $s_1 + a_\alpha s_2$ of $S_i$ have an $\alpha$-valuation of difference one. 
	
	We prove the claim by contradiction. Recall that if $(f + a_\alpha g)(\alpha) = 0$ and $v_\alpha(f + a_\alpha g) \neq 1$, we know $(f' + a_\alpha g')(\alpha) = 0$. Hence, assuming the claim not to hold, we find for infinitely many $\alpha$ that
	$$ a_\alpha(fg' - f'g)(\alpha) = 0.$$ 
	Since we know for all $\alpha$ that $a_\alpha \neq 0$, this would imply that $fg' = f'g$. However, because $\deg(f) > \deg(f')$, this contradicts the coprimality of $f$ and $g$. The claim and therefore the lemma now follow. 
\end{proof}

These ideas should give all the possibilities for $\deg(S_t)$ when $t  \geq \gamma_{t-1} + 4$. If $t$ does not meet this bound we will use the first part of Lemma \ref{Ldivisorgrowth} below to reduce the options for possible degrees of $S_{\gamma + 3}$.

\subsection{Tools for step 2 and 3}
The goal of step 2 is to find relations between the $e_ie_j$ such that these can be used in step 3 to find the poles of $e_1, \ldots, e_r$ outside of infinity. Here Theorem \ref{Tknownbasis} is also of big help, for example, if $\gamma_i = 0$ for some $i\geq 2$ we know that $e_i = e_2^{i-1}$. Hence, it is most of all important to find a relation that expresses $e_t$ in the other basis elements. When $t \geq \gamma + 3$ we know such a relation must exist, because the co-dimension of $S_{t-1}^2$ in $S_t^2$ is at most $\gamma + 2$ and there are $t$ elements of the form $e_1e_t, \ldots, e_te_t$. When $t < \gamma + 3$, we might need to look at a basis for $S_{t+1}^2$ or for $S_{t+2}^2$, \Change{for which Lemma \ref{Ldivisorgrowth} below will also be needed.}

We first introduce the following definitions to be able to better compare bases of several spaces. 
\begin{td}
	For $k, l \geq 1$, we define $E_{k, l} := \{e_ie_j \mid 1 \leq i \leq k, 1 \leq j \leq l\}$ and $E_k := E_{k, k}$.
\end{td}
\begin{td}
	Let $1 \leq i, j \leq n$ and $S' \subset S_iS_j$ some $K$-vector subspace. \Change{Define $S''$ such that $S_iS_j = S' \oplus S''$. } We say that \textit{$k$ elements from $T \subset E_{i, j}$ are needed for $S_iS_j$ compared to $S'$}, if every basis of \Change{$S''$ } that is a subset of $E_{i, j}$, contains exactly $k$ elements from $T$. If $T$ equals the singleton $\{s\}$ and $k=1$, we also say that \textit{$s$ is needed}. Furthermore, when saying something is needed for the basis of $S_i^2$ without further specifying $S'$, we always mean compared to $S_{i-1}^2$.
\end{td}

For example, for the basis of $S_t^2$ exactly $\gamma - \gamma_{t-1} + 2$ elements from $E_t$ are needed. When we know which $\gamma - \gamma_{t-1} + 2$ elements are needed, we also know which ones are not (necessarily) needed, giving us a relation between the $e_ie_j$ in $E_t$. To find out which elements are needed we can use a degree table of $S_t^2$ (see Table \ref{TABex} for an example) and ideas similar to the following proposition. 

\begin{tp}\label{Pneeded}~
	\begin{enumerate}[\normalfont(i)]
		\item Let $1 \leq i \leq t$ and suppose $\deg(e_{i}e_t) > \max(\deg(S_{t-1})^2)$, then $e_ie_t$ is needed for the basis of $S_t^2$. 
		\item When $e_ie_t$ is not necessarily needed for the basis of $S_t^2$, there must either be an element $f_1 \in E_{t-1}$ such that $\deg(e_ie_t) = \deg(f_1)$ or two distinct elements $f_1 \in E_{t-1}$ and $f_2 \in E_t$ such that $\deg(f_1) = \deg (f_2) > \deg(e_ie_t)$.
	\end{enumerate}
\end{tp}
\begin{proof}
	Let $s_1, \ldots, s_k \in E_t$ such that $S_t^2 = \langle S_{t-1}^2, s_1, \ldots, s_k \rangle$ and $k$ is minimal. Then we should be able to find $a_0, \ldots, a_k \in K$ and $s_0 \in S_{t-1}^2$ such that
	$e_ie_t = \sum_{j=0}^k a_js_j$. Clearly, the degree of $e_ie_t$ should be equal to the degree of the right hand side. Hence, either $\deg(e_ie_t) = \deg(s_j)$ for some $0 \leq j \leq k$, or $\deg(e_ie_t) < \deg(s_{j_1}) = \deg(s_{j_2})$ for some $0 \leq j_1 < j_2 \leq k$. In the second case we must have that $j_1 = 0$, because the degree of $e_{j_1}e_t$ is never equal to the degree of $e_{j_2}e_t$ if $j_1 \neq j_2$. This proves both parts of the proposition.
\end{proof}
\begin{te}
	Using the degree table shown in Table \ref{TABex} it follows from (i) that $e_4e_5$ and $e_5^2$ are needed for the basis of $S_5^2$ and from (ii) that $e_3e_5$ is also needed. Assuming that exactly three elements from $\{e_ie_5 \mid 1 \leq i \leq 5\}$ are needed, we could now conclude that $e_2e_5$ and $e_1e_5$ must be in $S_4^2$. Hence, $e_2e_5 = \sum_{1 \leq i \leq j \leq 4} a_{i, j}e_{i}e_j$ for some $a_{i, j} \in K$, which is a non-trivial relation between the $e_i$'s. 
\end{te}

	\begin{table} [h] 
	\begin{center}
		\begin{tabular}{c | c | c | c | c | c}
			&$e_1$&$e_2$&$e_3$&$e_4$&$e_5$\\\hline
			$e_1$& $0$ & $1$& $2$ & {$4$}&$5$\\\hline
			$e_2$& - & $2$ & $3$& $5$& $6$ \\\hline
			$e_3$& - & - & $4$& $6$&$7$ \\\hline
			$e_4$& - & - & -& $8$& $9$ \\\hline
			$e_5$& - & - & -& -& $10$ \\
		\end{tabular}
		\caption{The degree table of $S_5^2$ when the degrees in $S_5$ equal $0, 1, 2, 4, 5$. In each cell the degree of the product $e_{i}e_j$ is written.}\label{TABex}
	\end{center}
\end{table}

To conclude something about the poles of $e_t$ it is useful to know the exact structure of $S_{t-1}^2$. The following can  be derived from Theorem \ref{Tknownbasis}.

\begin{tl}~
	\begin{enumerate}[\normalfont(i)]
		\item When $\gamma = 0$ and $n \geq 3$ we have $S_i = \mathfrak p_{i-1}$ and $S_{i}^2 = \mathfrak p_{2i-2}$ for all $1 \leq i \leq n$. 
		\item When $\gamma = 1$ and $n \geq 4$ the case $t_1 = t-1$ gives, $S^2 = \mathfrak p_{2n}$ and $S = \mathfrak p_{n-2} \oplus Ke_n$. (All $S_i$ with $i < n$, have $\gamma_i = 0$ and thus fall in the previous case.)
		\item When $\gamma = 1$ and $n\geq 4$ the case $t_1 = 3$ gives, $S_i = K \oplus e_2\mathfrak p_{i-2}$ and $S_i^2 = K \oplus e_2\mathfrak p_{2i-2}$ for all $3 \leq i \leq n$ and $\frac{e_i}{e_{i-1}} \in K[x]$ for all $2 \leq i \leq n$. 
	\end{enumerate}
\end{tl}

Furthermore, we will make use of the following lemma.

\begin{tl}\label{Lpoles}
	The number of poles of $e \in F$ counted with multiplicity equals $[F:K(e)]$. 
\end{tl}
	\Change{
\begin{proof}
	Let $x \in F$ such that $F = K(x)$, then $e = \frac{f(x)}{g(x)}$ with $f(x), g(x) \in K[x]$ coprime. 
		This gives us the algebraic relation $f(x) - g(x)e = 0$. The polynomial $f(x) - g(x)e$ is irreducible in $K[e][x]$, and therefore, by Gauss's Lemma, also in $K(e)[x]$. Hence, 
		\[ [F:K(e)] = \max(\deg(g(x), \deg(f(x))) = \textup{ number of poles of $e$ with multiplicity}. \qedhere \]
\end{proof}}
\subsection{Tools for Step 4}
When we know all the poles occurring at $e_1, \ldots, e_r$ we know exactly what the divisor $D_i$ of least degree such that $S_i \subset L(D_i)$ looks like for $1 \leq i \leq r$. However, it is possible that $n > r$. The following lemma helps us determine $D_{i+1}$ when $\gamma_{i+1}$ stays equal to $\gamma_i$. 
	\begin{tl}\label{Ldivisorgrowth}
	Suppose $\gamma_{i+1} = \gamma_{i}$ then, 
	\begin{enumerate}[\normalfont(i)]
		\item 	if $i\geq 2$ we have $D_{i+1} - D_i = D_{i} - D_{i-1}$.
		\item  if $i \geq 3$ and $\alpha \in K\cup \{\infty\}$ such that $v_{\alpha}(D_{i-1}) = 0$ we have $v_{\alpha}(D_i) = 0$.
	\end{enumerate}
\end{tl}
\begin{proof}
	We start by proving the first statement. Note that the co-dimension of $S_{i-1}S_i$ inside $S_i^2$ is $1$, since $e_i^2$ is the only element of $E_i$ not in $S_{i-1}S_i$. Because $\gamma_{i+1} = \gamma_i$, we also know that the co-dimension of $S_i^2$ inside $S_iS_{i+1}$ equals one. Hence, we find a co-dimension diagram as shown in Figure \ref{Fallisone}, \Change{where the number on an arrow from some space $A$ to some space $B$ indicates the co-dimension of $A$ in $B$.}
		\begin{figure}[h]
		\begin{center}
			\begin{tikzcd}[row sep=6ex, column sep=3ex]
				S_i^2 \arrow[r, "1"]                       & S_iS_{i+1}                             \\
				S_{i-1}S_{i} \arrow[u, "1"] \arrow[r, "1"] & S_{i-1}S_{i+1} \arrow[u, "1"]
			\end{tikzcd}\caption{The co-dimensions are all equal to $1$.}\label{Fallisone}
		\end{center}
	\end{figure}\\
	This implies that 
	$$S_iS_{i+1} = \langle S_{i-1}S_{i+1}, e_ie_{i+1}, e_i^2 \rangle = \langle S_i^2, e_{i}e_{i+1} \rangle,$$
	and thus $e_i^2 \in S_{i-1}S_{i+1}$. We find that $S_i^2 \subset S_{i-1}S_{i+1}$ and since they have the same dimension it must hold that $S_i^2 = S_{i-1}S_{i+1}$.  

	For any $\alpha \in K\cup\{\infty\}$ it now holds that $\min v_\alpha(S_i^2) = \min v_{\alpha}(S_{i-1}S_{i+1})$ and thus,
	$$\min v_{\alpha}(S_{i+1}) - \min v_{\alpha}(S_i) = \min v_{\alpha}(S_i) - \min v_\alpha(S_{i-1}).$$
	Recalling that $v_\alpha(D_j) = -\min v_\alpha(S_j)$ for any $1 \leq j \leq n$, we can conclude the first statement.
	%		We prove this for every pole $\alpha \in K\cup\{\infty\}$ separately. Translate the basis $(e_1, \ldots, e_{i+1})$ of $S_{i+1}$ such that $v_\alpha(e_1) \geq v_\alpha(e_2) \geq \ldots \geq v_\alpha(e_{i+1})$, this is always possible in such a way that $\deg(e_j)$ stays the same for each $1 \leq j \leq i+1$. In particular we get that $v_{\alpha}(e_j) = -v_\alpha(D_j)$. 
	%		
	%		Consider $S_{i-1}S_{i+1}$, then $e_{i-1}e_{i+1}$ has highest degree and most number of poles at $\alpha$ and therefore is certainly needed in the basis of $S_{i-1}S_{i+1}$. The same holds for $e_i^2$ in $S_{i}^2$. However, they are never both in a basis of $S_{i+1}^2$ and since $e_ie_{i+1}$ and $e_{i+1}^2$ definitely have a higher degree than both, there must be some relation 
	%		$$e_{i-1}e_{i+1}  = e_i^2 + \sum_{1 \leq j \leq i - 1, 1 \leq k \leq i} a_{j, k}e_je_k + \sum_{1 \leq j \leq i-2} a_{j, i+1}e_je_{i+1}.$$ 
	%		This shows $v_\alpha(e_{i-1}e_{i+1}) \geq v_\alpha(e_i)^2$, with equality when $v_\alpha(e_{i-1}) > v_\alpha(e_i)$. When $v_\alpha(e_{i-1}) = v_\alpha(e_i)$ we get $v_\alpha(e_{i+1}) \geq v_\alpha(e_i)$, but we know this can only be an equality by how we chose our basis. Hence, we always get $v_\alpha(e_{i+1}) - v_\alpha(e_{i}) = v_\alpha(e_i) - v_\alpha(e_{i-1})$. 
	
	Suppose $v_\alpha(D_i) = s > 0$ then, by the first statement, $e_{i+1}$ has $2s$ poles at $\alpha$, which is more than any element in $S_{i-1}S_i$ has. We know $e_{i+1}$ cannot be needed for the basis of $S_{i+1}^2$, since $e_ie_{i+1}$ and $e_{i+1}^2$ are already needed. This implies that $\deg(e_{i+1}) = \deg(e_i^2)$, which contradicts that $\deg(e_{i+1}) - \deg(e_i) = \deg(e_i) - \deg(e_{i-1})$. 
\end{proof}	
\begin{tr}
	The first part of Lemma \ref{Ldivisorgrowth} can be seen as a generalization of Lemma 7.10 in \cite{bachoc2018funcfieldfreiman}.
\end{tr}

\section{From $0$ to $\gamma$} \label{Sgam0}
In this section we will prove Theorem \ref{Tmain} when $\{\gamma_i \mid 1 \leq i \leq n\} = \{0, \gamma\}$, or in other words, the sequence $\gamma_i$ jumps directly from $0$ to $\gamma$. We also assume that $\gamma \geq 2$, which we may do because the conjecture is already proved for $\gamma = 1$. As mentioned before, we know that $t \geq \gamma -\gamma_{t-1} + 2 = \gamma + 2$ and we will consider the cases I. $t \geq \gamma + 3$ and II. $t = \gamma + 2$ separately. 

\subsubsection*{I. The case $t \geq \gamma + 3$}
The first step is to determine all the possible degrees of $e_1, \ldots, e_t$, which is done by the following proposition. 
\begin{tp}
	The degree of $e_i$ for $1 \leq i \leq t-1$ is equal to $i-1$ and the degree of $e_t$ is equal to $t-2 + \Delta$ for some integer $\Delta \in [1, \gamma + 1]$. 
\end{tp}
\begin{proof}
	By Lemma \ref{Lboundedfieldextension} we know that $F = K(S_{\gamma + 2}) = K(S_{t-1})$. Hence, the degree of $e_i$ for $i \leq t-1$ is determined by its degree in $K(S_{t-1})$. Since we are given that $\gamma_{t-1} = 0$ and $t-1 \geq 3$, we can use Theorem \ref{Tknownbasis} to conclude $S_{t-1} \subset L((t-2)P_\infty)$ and therefore $\deg(e_i) = i-1$ for all $1 \leq i \leq t-1$. The second statement now follows directly from Freiman's $3k-4$ Theorem. 
\end{proof}
Next we want to find relations between the basis elements of $S_t$ to be able to determine their poles. These two steps both occur in the proof of the next proposition. 
\begin{tp}\label{Pgamma0gen}
	The divisor $D_t$ takes the form $(t-2+\Delta)P_\infty + P_{\alpha_1} + \ldots + P_{\alpha_{k}}$ with $k  = \gamma - \Delta + 1$ and $\alpha_i \in K$ not necessarily distinct. 
\end{tp}
\begin{proof}
	There are exactly $\gamma + 2$ elements needed from $\{e_ie_t \mid 1\leq i \leq t\}$ for the basis of $S_t^2$. By Proposition \ref{Pneeded}(i) we know that $e_ie_t$ is always needed for $i \leq t - 1 - \Delta$, giving us $\Delta + 1$ elements. Furthermore, there must be at least one element that is not needed. We define $k$ to be the smallest positive integer such that $e_{k+1}e_t$ is an element of 
	\begin{equation} \label{EQspace}
		\langle S_{t-1}^2, e_1e_t, \ldots, e_{k}e_t\rangle, 
	\end{equation}
	then $k+1 < t $. This implies that there exists $a_1, \ldots, a_{k+1} \in K$ such that
	\[	\left(\sum_{i=1}^{k+1} a_ie_i\right)e_t \in S_{t-1}^2 = \mathfrak p_{2t-4}.	 \]
	Since $e_i$ has exactly degree $i-1$, we find that $e_t$ can have at most $k$ poles outside of infinity. Furthermore, because $k$ is defined as the smallest integer for which such a relation holds, we find that $e_t$ must have exactly $k$ of such poles. This implies that any element $e_ie_t$ of degree smaller or equal to $2t-4$ must be in (\ref{EQspace}) hence, whenever $i \leq t - 1-\Delta$. Hence, we get $k$ basis elements for small $i$ and $\Delta + 1$ for big $i$ and find $\Delta + k + 1  = \gamma + 2$ or, equivalently, $k = \gamma - \Delta + 1$. This shows that $D_t = (t-2+\Delta)P_\infty + P_{\alpha_1} + \ldots + P_{\alpha_{k}}$. 
\end{proof}

In the last step we find that $n$ must always equal $t$ and may conclude that the conjecture does indeed hold in this case.

\begin{tp}\label{Pfirsttisn}
	We have $n = t$ and thus $S \subset L(D)$ with $\deg D = n + \gamma - 1$.
\end{tp}
\begin{proof}
	We prove this by contradiction, so assume $n > t$ then we must have the basis element $e_{t+1}$ and $\gamma_{t+1} = \gamma_t$. By Lemma \ref{Ldivisorgrowth} we find that $k$ must equal $0$ and that $e_{t+1}$ is of degree $t-2+2\Delta = t + 2\gamma$. This contradicts Freiman's $3k-4$ Theorem whenever $\gamma \geq 2$. 
\end{proof}

\subsubsection*{II. The case $t = \gamma + 2$}	
We may immediately consider the space $S_{t+1} = S_{\gamma + 3}$, since we know $n \geq \gamma + 3$. The first step is determining the degrees of $e_1, \ldots, e_{t+1}$. 
\begin{tp}
	The set of degrees of $S_{\gamma + 3}$ equals either $\{0, 1, 2, \ldots, \gamma + 2\}$ or \\$\{0, 2, 4, \dots, 2\gamma, 2\gamma+1, 2\gamma + 2\}$. 
\end{tp}
\begin{proof}
	By Lemma \ref{Ldivisorgrowth} we may define positive integers $d$ and $k$ such that the degree of $e_i$ is $(i-1)d$ when $i \leq \gamma + 1$ and $\gamma d + (i-\gamma - 1)k$ when $i = \gamma + 2, \gamma + 3$. By Lemma \ref{Lgcd} we find that either $d$ or $k$ has to equal $1$. This gives us the degree table as shown in Table (\ref{TABgamma0}). 
	
		\begin{table} [h] 
		\begin{center}
			\begin{tabular}{c | c | c | c | c | c | c | c}
				&$e_1$&$e_2$&$e_3$& $\ldots$ &$e_{t-1}$& $e_{t}$ & $e_{t+1}$ \\\hline
				$e_1$& $0$ & $d$& $2d$ & $\ldots$ &$\gamma d$ & $\gamma d + k$ & $\gamma d + 2k$ \\\hline
				$e_2$& - & $2d$ & $3d$& $\ldots$ & $(\gamma + 1)d$ & $(\gamma+1)d + k$ & $(\gamma + 1)d + 2k$ \\\hline
				$e_3$& - & - & $4d$& $\ldots$ & $(\gamma + 2)d$ & $(\gamma + 2)d + k$ & $(\gamma + 2)d + 2k$ \\\hline
				$\vdots$& $\vdots$ & $\vdots$ & $\vdots$& $\ddots$ & $\vdots$ & $\vdots$ & $\vdots$ \\\hline
				$e_{t-2} $& - & - & -& $\ldots$ & $(2\gamma - 1)d$ & $(2\gamma - 1)d + k$ & $(2\gamma - 1) d+ 2k$ \\\hline
				$e_{t-1}$ & - & - & - & $\ldots$ & $2\gamma d$ & $2\gamma d + k$ & $2\gamma d + 2k$\\\hline
				$e_{t}$ &-&-&-&$\ldots$&-&$2\gamma d + 2k$ & $2\gamma d + 3k$\\\hline
				$e_{t + 1}$ & - & - & - & $\ldots$ & -& -& $2\gamma d + 4k$\\
			\end{tabular}
			\caption{The degree table of $S_{t + 1}^2$. In each cell the degree of the product $e_{i}e_j$ is written.}\label{TABgamma0}
		\end{center}
	\end{table}
	Because $\gamma_t = \gamma_{t+1}$ we know that exactly two elements of $E_{t+1}$ are needed for the basis of $S_{t+1}^2$, namely $e_{t}e_{t+1}$ and $e_{t+1}^2$. In particular this implies that $e_{t-2}e_{t+1}$ should be in the space $S_t^2$. By Proposition \ref{Pneeded}(ii) either two elements in $E_{t}$ have an equal degree greater than $\deg(e_{t-2}e_{t+1})$ or $\deg(e_{t-2}e_{t+1})$ is in $\deg(S_t^2)$. The first option cannot happen because $2\gamma d, 2\gamma d + k$ and $2 \gamma d + 2k$ are all distinct. We find that one of the following two equations must hold
	\[(2\gamma - 1)d + 2k = 2\gamma d ~~~~~ \textup{ or } ~~~~~ (2\gamma - 1)d + 2k = 2\gamma d + k.\]
	We conclude that $(d, k)$ is equal to either $(2, 1)$ or $(1, 1)$. The lemma now follows.
\end{proof}	

We will now distinguish between the two possible sets of degrees that $S_{\gamma + 3}$ can have and determine the possible poles for both cases. 

\begin{tp}
	Suppose the set of degrees of $S_{\gamma + 3}$ equals $\{0, 2, 4, \ldots, 2\gamma, 2\gamma + 1, 2\gamma + 2\}$. Then $S \subset L((n + \gamma - 1)P_\infty)$. 
\end{tp}
\begin{proof}
	The fact that $\gamma_{t-1} = 0$ and the degree table of $S_t^2$ (Table \ref{TABgamma0} with $d = 2$ and $k=1$), will help us to find relations between the $e_i$. Note that we can find a basis for $S_{t+1}^2$ consisting of elements in $E_{t+1}$ such that all degrees are distinct and therefore $e_{t+1} \in \langle S_2S_{t-1}, e_t \rangle$ and  $e_{t-1}e_{t+1} \in S_t^2$. By also possibly translating $e_{t+1}$ with other basis elements $e_i$ we find the following relations:
	\begin{align*}
		&e_i = e_2e_{i-1} = e_2^{i-1} ~~ \textup{for all $3 \leq i \leq t-1$} \\
		&e_{t+1} = e_2e_{t-1} = e_2^{t- 1} \\
		&e_{t-1}e_{t+1} = e_2^{2t-3} = e_{t}^2 + \sum_{\substack{1 \leq i \leq t-1 \\ i \leq j \leq t}} a_{i, j}e_ie_j,
	\end{align*}
	for some $a_{i, j} \in K$. The first relation implies that $F = K(e_2, e_{t})$ and the third that both $[F:K(e_2)] \leq 2$ and $[F:K(e_{t})] \leq 2\gamma + 1$. From Lemma \ref{Lpoles} we may conclude that $e_2$ and $e_{t}$ do not have any poles outside of infinity and therefore none of the $e_i$ with $i \leq t+1 = \gamma + 3$ do. The result now follows by Lemma \ref{Ldivisorgrowth}(i).
\end{proof}

\begin{tp}
	Suppose the set of degrees of $S_{\gamma + 3}$ equals $\{0, 1, 2, \ldots, \gamma + 2\}$. Then, $S \subset L((n-1)P_\infty + \gamma P_{\alpha})$ for some $\alpha \in K$.
\end{tp}
\begin{proof}
	As in the previous proof, but now filling in $d=1$ and $k=1$ in the degree table, we find the following relations
	\begin{align*}
		&e_i = e_2e_{i-1} = e_2^{i-1} ~~\textup{for all } 3 \leq i \leq t-1\\
		&e_{t-2}e_{t+1} = e_{t-1}e_{t} + \sum_{\substack{1 \leq i \leq j \leq t \\ i+j \leq 2t-2}} a_{i, j} e_ie_j \\
		&e_{t-1}e_{t+1} = e_{t}^2 + \sum_{\substack{1 \leq i \leq t- 1 \\ i \leq j \leq t}} b_{i, j} e_ie_j,
	\end{align*}
	for some $a_{i, j}, b_{i, j} \in K$. The first equation implies that $F = K(e_2, e_{t})$. When we multiply the second equation by $e_2$, its left side is equal to the left side of the third equation. Therefore, the right sides must also be equal, which gives us an equation in only $e_2$ and $e_{t}$; we find $[F:K(e_2)] \leq 2$ and $[F : K(e_{t})] \leq 2\gamma + 1$. 
	
	Suppose that $e_2$ does not have any poles outside of infinity, then by Lemma \ref{Ldivisorgrowth} neither does $e_{t}$. However, then $S_{t+1}$ consists exactly of all the polynomials of degree less or equal to $\gamma + 2$ and hence, $\gamma_{t+1} = 0$. We may conclude that $e_2$ has exactly one pole outside of infinity, at $\alpha$.
	
	This shows that $e_{t-1}$ must have exactly $\gamma$ poles outside of infinity, all at $\alpha$. By possibly translating $e_{t}$ by a multiple of $e_{t-1}$, we also find that $e_{t}$ has $\gamma$ poles at $\alpha$. Since $e_{t}$ can have at most $2\gamma + 1$ poles these must be the $\gamma + 1$ poles at infinity and the $\gamma$ poles at $\alpha$. Hence, $D_{\gamma + 2} - D_{\gamma + 1} = P_\infty$ and the statement now follows from Lemma \ref{Ldivisorgrowth}.
\end{proof}

\section{From 0 to 1 to $\gamma$} \label{Sgam01}
In this section we will prove Theorem \ref{Tmain} for the second case, when $\{\gamma_i \mid 1 \leq i \leq n\} = \{0, 1, \gamma\}$. We know that $t \geq \gamma + 1$ and, by Theorem \ref{Tknownbasis}, that whenever $t \geq 5$ the only options for $t_1$ are $3$ and $t-1$. It therefore suffices to consider the following four cases: I. $t \geq \gamma + 3$ and $t_1 = t - 1$, II. $t \geq \gamma + 3$ and $t_1 = 3$, III. $t \in \{\gamma + 1, \gamma + 2\}$ and $t \geq 5$ and IV. $t = 4$. 

The next proposition determines the degrees of $e_{t-1}$ and $e_t$ in the first three cases, but does not necessarily hold in the fourth.
\begin{tp}
	Suppose $t \geq \max(\gamma+1, 5)$, then the degree of $e_{t-1}$ equals $t-1$ and the degree of $e_t$ equals $t-1 + \Delta$ for some $\Delta \in [1, \gamma]$. 
\end{tp}
\begin{proof}
	Since, $t-1 \geq 4$ we know from Theorem \ref{Tknownbasis} that $\deg(\{e_1, \ldots, e_{t-1}\})$ equals either $\{0, d, 2d, \ldots, (t-3)d, (t-1)d\}$ or $\{0, 2d, 3d, \ldots, (t-1)d\}$ for some positive integer $d$. Hence, the degrees form an arithmetic progression of step $1$ missing $(t-1)(d-1) + 1$ elements. Using that $t-1 \geq \gamma$, Freiman's $3k-4$ Theorem now implies that $d = 1$, which proves the first two statements. The third also follows from Freiman's $3k-4$ Theorem, since $1 + (\Delta - 1)$ must not exceed $\gamma$. 
\end{proof}
Note that, given $t_1$, we have also found the degrees of $e_1, \ldots, e_{t-2}$.

\subsubsection*{I. The case $t \geq \gamma + 3$ and $t_1 = t - 1$}
Since $t-1 \geq 4$ we already know all the poles for $e_1, \ldots, e_{t-1}$ and in the next proposition the poles at $e_t$ are determined. 
\begin{tp}
	Let $k$ equal $\gamma - \max(\Delta, 2)$ and $\alpha_1, \ldots, \alpha_{k} \in K$, not necessarily distinct. The divisor $D_t$ takes the form $$(t-1+\Delta)P_\infty + P_{\alpha_1} + \ldots + P_{\alpha_{k}}.$$ 
\end{tp}
\begin{proof}	
	Exactly $\gamma + 1$ elements of $\{e_ie_t \mid 1 \leq i \leq t\}$ are needed for the basis of $S_t^2$. Consider the degree table of $S_t^2$, shown in Table \ref{TABgamma1I} and recall Proposition \ref{Pneeded}. Since the degree of $e_ie_t$ exceeds $2t-2$, the maximal degree of $S_{t-1}^2$, for all $i \geq t - \max(\Delta, 2) + 1$, these elements must all be needed for the basis of $S_t^2$, giving $\max(\Delta, 2)$ needed basis elements. Furthermore, because $E_{t-1}$ does not contain any elements of degree $2t-3$ and exactly one of degree $2t-2$, at least one of the elements $e_{t-\Delta}e_t$ and $e_{t-\Delta-1}e_t$ are needed. 

	Define $k$ to be the smallest integer such that $e_{k+1}e_t$ is an element of 
	\begin{equation*}
		\langle S_{t-1}^2, e_1e_t, \ldots, e_{k}e_t\rangle.
	\end{equation*}
	Since there are at least two elements of the form $e_ie_t$ that are not needed for the basis of $S_t^2$, we know that $k + 1 < t-\Delta - 1$ and therefore, $\deg(e_{k+1}e_t) < \deg(e_{t-1}^2)$. In particular this implies that $e_{k+1}e_t$ is an element of $\langle S_{t-2}S_{t-1}, e_1e_t, \ldots, e_ke_t \rangle$. Recall that $S_{t-2}S_{t-1} = \mathfrak p_{2t-4}$, hence there must exist a relation
	$$\left(\sum_{i=1}^{k+1} a_ie_i \right) e_t \in \mathfrak p_{2t-4},$$
	for some $a_i \in K$. As in the proof of Proposition \ref{Pgamma0gen} we can conclude that $e_t$ has exactly $k$ poles outside of infinity and that any $e_ie_t$ with degree less or equal to $2t-4$, can also be described by such a relation. By a similar argument we find that $e_{t-\Delta}e_t$ cannot be needed. Hence, $k + \max(\Delta, 2) + 1 = \gamma + 1$. This proves the proposition. 
\end{proof}
	\begin{table}[h]
	\begin{center}
		\begin{tabular}{c | c | c | c | c | c | c | c | c}
			&$e_1$&$e_2$&$e_3$& $\ldots$ &$e_{t-2}$& $e_{t-1}$ & $e_{t}$ & $e_{t+1}$ \\\hline
			$e_1$& $0$ & $1$& $2$ & $\ldots$ &$t-3$ & $t-1$ & $t-1+\Delta$ & $t-1+2\Delta$ \\\hline
			$e_2$& - & $2$ & $3$& $\ldots$ &$t-2$ & $t$ & $t+\Delta$ & $t+2\Delta$\\\hline
			$e_3$& - & - & $4$& $\ldots$ &$t-1$ & $t+1$ & $t+1+\Delta$ & $t+ 1 + 2\Delta$\\\hline
			$\vdots$& $\vdots$ & $\vdots$ & $\vdots$& $\ddots$ & $\vdots$ & $\vdots$ & $\vdots$ & $\vdots$ \\\hline
			$e_{t-2} $& - & - & -& $\ldots$ &$2t-6$ & $2t-4$ & $2t-4+\Delta$ & $2t-4+2\Delta$  \\\hline
			$e_{t-1}$ & - & - & - & $\ldots$ &- & $2t-2$ & $2t-2+\Delta$& $2t-2+2\Delta$\\\hline
			$e_{t}$ &-&-&-&$\ldots$&-&- & $2t-2+2\Delta$ & $2t-2+3\Delta$\\\hline
			$e_{t+1}$ & - & - & - & $\ldots$&-&-&-&$2t-2 + 4\Delta$ \\
		\end{tabular}
		\caption{The degree table of $S_{t+1}^2$ when $t_1 = t - 1$. In each cell the degree of the product $e_{i}e_j$ is written.}\label{TABgamma1I}
	\end{center}
\end{table}
\begin{tp}
	Suppose either $\gamma > 2$ or $\Delta > 1$ then, $n = t$ and thus, $S \subset L(D)$ with $\deg D = n + \gamma - 1$. When $\Delta = 1$ and $\gamma = 2$ it holds that $n\geq t$ and $S \subset L(nP_\infty)$. 
\end{tp}
\begin{proof}
	With some minor adjustments, the first statement can be proved in the same way as Proposition \ref{Pfirsttisn}. When $\gamma = 2$ and $\Delta = 1$ it follows from Lemma \ref{Ldivisorgrowth}.
\end{proof}

\subsubsection*{II. The case $t \geq \gamma + 3$ and $t_1 = 3$}
As in the previous case, we already know the poles for $e_1, \ldots, e_{t-1}$ and we start with determining the poles of $e_t$. In this proof we really need some of the exact structure of $S_{t-1}^2$ that comes from the explicit basis given in Theorem \ref{Tknownbasis}.
\begin{tp}
	The divisor $D_t$ takes the form 
	$$(t-1 + \Delta)P_\infty + P_{\alpha_1} + \ldots + P_{\alpha_{k}},$$
	for $k$ equal to either $\gamma - \Delta$ or $\gamma - \Delta - 1$ and $\alpha_1, \ldots, \alpha_k \in K$ not necessarily distinct.
\end{tp}
\begin{proof}
	By Proposition \ref{Pneeded}(i) we find that  all $e_ie_t$ with $i \geq t-\Delta$ are needed for the basis of $S_t^2$, this gives us $\Delta + 1$ basis elements. Since at least two elements of $\{e_ie_t \mid 1\leq i \leq t\}$ are not needed we know that for some $a_i \in K$ we have 
	$$\sum_{i=1}^{t - \Delta - 2} a_ie_ie_t \in S_{t-1}^2 \subset \mathfrak p_{2t-2},$$ 
	and therefore $e_t$ has at most $\deg(e_{t-\Delta - 2}) = t - \Delta - 2$ poles outside of infinity. 
	
	Define $k$ to be the exact number of poles that $e_t$ has outside of infinity, then for all $k+2 \leq i \leq t - \Delta - 1$ we can construct a polynomial of the form $A_i := \frac 1{e_2}(e_i + \sum_{j=2}^{k+1} a_{j, t}e_j)$ with $ a_{j, t} \in K$ while fixing $k$ of its zeroes. Taking the zeroes of $A_i$ equal to the poles of $e_t$ we find that $e_2A_ie_t \in e_2\mathfrak p_{2t-4}$ and thus 
	\begin{equation*}
		e_ie_t \in \langle S_{t-1}^2, e_1e_t, \ldots, e_{k+1}e_t \rangle,
	\end{equation*}
	for all $i \leq t - \Delta - 1$. Hence the basis of $S_t^2$ needs at most $k+ \Delta + 2$ elements from $\{e_ie_t \mid 1\leq i \leq t\}$ for its basis and $k \geq \gamma - \Delta - 1$. Furthermore,
	\begin{equation*}
		e_ie_t \not\in \langle S_{t-1}^2, e_1e_t, \ldots, e_{i-1}e_t \rangle,
	\end{equation*}
	for all $1 \leq i \leq k-1$, because then there would exist a polynomial $h$ of degree smaller or equal to $k-1$ such that $he_t$ would not have any poles outside of infinity. This shows that $S_t^2$ needs at least $k + \Delta$ elements from $\{e_ie_t\}$ for its basis and therefore $k \leq \gamma - \Delta + 1$. 
	
	To prove the proposition it now suffices to show that $k \neq \gamma - \Delta + 1$. We will prove this by contradiction, so assume $e_t$ has exactly $\gamma - \Delta + 1$ poles outside of infinity. By our two earlier observations we see that $\langle S_{t-1}^2, e_1e_t, \ldots, e_{k-1}e_t, e_{t-\Delta}e_t, \ldots, e_t^2\rangle = S_t^2$. In particular, this implies that 
	\begin{equation}
		e_ke_t, e_{k+1}e_t \in \langle S_{t-1}^2, e_1e_t, \ldots, e_{k-1}e_t \rangle.
	\end{equation}
	Thus, we can construct polynomials $A_{i} := e_i + \sum_{j=1}^{k-1} a_{i, j}e_j$ for $i=k, k+1$ and $a_{i, j} \in K$ such that $A_ke_t, A_{k+1}e_{t} \in S_{t-1}^2 = K \oplus e_2\mathfrak p_{2t-4}$. This means that we can write $A_{k+1}e_t = e_2f + a$ and $A_ke_{t} = e_2g + b$ for some polynomials $f, g \in K[x]$ and $a, b \in K$. Multiplying the second equation by $\frac {A_{k+1}}{A_k}$ we find that
	$$e_2(f - \frac{A_{k+1}}{A_k}g) = \frac{A_{k+1}}{A_k}b - a.$$
	The degree of $\frac{A_{k+1}}{A_k}$ equals $1$ and hence, by comparing degrees, we derive that $a = b = 0$. We can now rewrite our earlier equation to $\frac{A_k}{e_2}e_t = g$. However, this implies that $e_t$ has at most $k-2$ poles outside of infinity, which contradicts our assumptions. 
	%	Then there must exist some relation 
	%	$$\sum_{j=1}^k a_je_je_t \in S_{t-1}^2 = K \oplus e_2 \mathfrak p_{2t-4} \subset \mathfrak p_{2t-2}.$$
	%	The degree of $e_k$ is $k$ and therefore, $e_t$ can have at most $k$ poles outside of infinity. Note that $e_t$ cannot have $k-3$ or less poles outside of infinity, because then $k$ would not have been the smallest integer for which (\ref{EQsubspace}) holds. 
	%	
	%	Assume that $e_t$ has exactly $k$ poles outside of infinity, then they are the zeroes of $A_k := \sum_{j=1}^k a_je_j$ which is a polynomial of degree $k$. For any $k\leq i < t - \Delta$ we can now find $a_{i, j} \in K$ such that 
	%	$$A_ie_t := \left( \left(\sum_{j=1}^{k-1} a_{i, j}e_j\right)+a_{i, i}e_i\right)e_t \in K \oplus e_2 \mathfrak p_{2t-4}.$$ 
	%	
\end{proof}
\begin{tp}
	Suppose either $\gamma > 2$ or $\Delta > 1$ then, $n= t$ and thus, $S \subset L(D)$ with $\deg D \in \{n + \gamma - 1, n + \gamma - 2\}$. When $\gamma = 2$ and $\Delta = 1$ it holds that $n \geq t$ and $S \subset L(nP_\infty)$. 
\end{tp}
\begin{proof}
	When $n > t$ we know from Lemma \ref{Ldivisorgrowth} that $k = 0$ and thus $\Delta = \gamma$ or $\Delta = \gamma -1$.  In the first case the degree of $e_{t+1}$ equals $t-1+2\gamma$, which contradicts Freiman's $3k-4$ Theorem for $\gamma \geq 2$. In the second case the degree becomes $t-3 + 2\gamma$, which only contradicts Freiman's $3k-4$ Theorem for $\gamma \geq 4$. This proves the proposition for all cases except $\gamma = 3$ and $\Delta = 2$. However, in that case it is easy to check $e_{t-2}e_{t+1} \not\in S_{t}^2$ and thus $\gamma_{t+1} > \gamma_t$, also a contradiction. 
\end{proof}

\subsubsection*{III. The case $t \in \{\gamma + 1, \gamma + 2\}$ and $t \geq 5$}
For this case the structure of the proof will be quite different from the other cases, as we will actually show that it can never hold. We will consider the options $F = K(e_2, e_3)$ and $F \neq K(e_2, e_3)$ separately in the next two propositions. In both proofs we will still roughly follow steps 2 and 3, in the sense that we try to find relations between the $e_i$ for $1 \leq i \leq \gamma + 3$, which will give us information about which poles they should have. However, we will see that this then implies that $\gamma_{\gamma + 3} > \gamma_t = \gamma$, which is not possible. 
\begin{tp}\label{Pnotfield}
	There exists no $S$ for which $F = K(e_2, e_3)$.
\end{tp}
\begin{proof}
	First we restrict the possible values that $\Delta$ can take now that we know $n \geq t + 1$. Consider the element $e_{t-2}e_{t+1}$, which has degree $2t - 3 + 2\Delta$ or $2t - 4 + 2\Delta$, when $t_1$ equals $3$ or $t-1$ respectively. Because $e_{t}e_{t+1}$ and $e_{t+1}^2$ are always the two elements needed for the basis of $S_{t+1}^2$, the element $e_{t-2}e_{t+1}$ must be in $S_t^2$. We find that $\deg(e_{t-2}e_{t+1}) \leq \deg(e_{t-1}e_t) = 2t - 2 + \Delta$. Hence $\Delta$ must equal either $1$ or $2$, where the second can only occur when $t_1 = t-1$.  
	
	Now assume that $F = K(e_2, e_3)$, all poles in $S_{t-1}$ may be assumed to be at infinity because, $t-1 \geq 4 = \gamma_{t-1} + 3$. Since $n > t$, Lemma \ref{Ldivisorgrowth}(ii) now implies that also all poles in $S$ are at infinity. 
	
	Recall that exactly $\gamma + 1$ elements of $\{e_ie_t \mid 1 \leq i \leq t\}$ are needed for the basis of $S_t^2$ or in other words, at most one is not needed. In particular this means that either $e_1e_t$ or $e_2e_t$ is needed. When $t_1 = t - 1$ we know that $S_{t-1}^2 =\mathfrak p_{2t-4}$, hence an element with no poles outside of infinity can only be needed for the basis of $S_t^2$ if its degree is bigger than $2t-4$. This would imply that $\deg(e_2e_t) = t + \Delta > 2t-4$, impossible when $\Delta = 1, 2$ and $t \geq 5$. When $t_1 = 3$, the degree of $e_ie_t$ might be included in $\deg(S_{t-1}^2)$ but still $e_ie_t \not \in S_{t-1}^2$, when $1 \in \deg(\langle S_{t-1}^2, e_ie_t \rangle)$. However, this can only happen once and thus two of $\deg(e_1e_t), \deg(e_2e_t)$ and $\deg(e_3e_t)$ must not be in $\deg(S_{t-1}^2)$. This also gives a contradiction for $\Delta = 1$ and $t \geq 5$. We conclude that $F \neq K(e_2, e_3)$, as wished.
\end{proof}

\begin{tp}
	There exists no $S$ for which $F \neq K(e_2, e_3)$. 
\end{tp}
\begin{proof}
 First we show that $F \neq K(e_2, e_3)$ can only happen if $t = \gamma + 1$. By Lemma \ref{Lboundedfieldextension} it holds that $F = K(e_2, e_3, e_{\gamma + 1})$, thus Proposition \ref{Pnotfield} implies that $e_{\gamma+1} \not \in K(e_2, e_3)$. When $t = \gamma + 2$ we have $\gamma_{\gamma+1} = 1$ and therefore $F(S_{\gamma+1}) = F(S_3) = F(e_2, e_3)$, a contradiction. From now on we assume $t = \gamma + 1$.

	Let $y \in F$ such that $K(e_2, e_3) = K(y)$ and for $1 \leq i \leq t-1$ we can write $e_i$ as a polynomial in $y$. From our previous observation we know that $e_{t} \not\in K(y)$ and $F = K(y, e_t)$.
	
	When $t_1 = t-1$, following the same argument as in the proof of Proposition \ref{Pnotfield}, we know $\Delta$ equals either $1$ or $2$. However, when $\Delta = 2$ we find that $e_{t-3}e_{t+1}$ is needed for the basis of $S_{t+1}^2$ by Proposition \ref{Pneeded}(ii), contradicting $\gamma_{t+1} = \gamma_t$. We conclude that $\Delta = 1$ and using Table \ref{TABgamma1I} we can find the following relations:
	\begin{align*}
		e_{t-2}e_{t+1} &= e_{t-1}^2 + \sum_{\substack{1 \leq i \leq t-2 \\ i \leq j \leq t}} a_{i, j}e_ie_j, \\
		e_{t-3}e_{t+1} &= e_{t-2} e_{t} + \sum_{\substack{1 \leq i \leq j \leq t \\ i+j \leq 2t-3}} b_{i, j}e_{i}e_j,
	\end{align*}
	for $a_{i, j}, b_{i,j} \in K$. We know that $e_{t-2} = ye_{t-3}$, thus when multiplying the second relation with $y$ the left hand side is the same as in the first relation. From this we find $[F : K(y)] = 1$, a contradiction. 
	
	When $t_1 = 3$ we know that $\Delta = 1$ and find the following relations using Table \ref{TABgamma1III},
	\begin{align*}
		e_{t-1} e_{t+1} &= e_{t}^2 + \sum_{\substack{1 \leq i \leq t-1 \\ i \leq j \leq t}} a_{i, j}e_ie_j, \\
		e_{t-2}e_{t+1} &= e_{t-1} e_{t} + \sum_{\substack{1 \leq i \leq j \leq t \\ i+j \leq 2t-2}} b_{i, j}e_{i}e_j,
	\end{align*}
	for some $a_{i, j}, b_{i, j} \in K$. We also know from Theorem \ref{Tknownbasis} that $f := \frac{e_{t-1}}{e_{t-2}}$ is a polynomial of degree $1$ in $y$. Multiplying the second relation with $f$ and comparing its right hand side with the right hand side of the first relation we find $[F : K(y)] \leq 2$ and $[F : K(e_{t})] \leq 2\gamma + 1$. Since $F \neq K(y)$, this implies that $y$ must have exactly one pole outside of infinity, at $\alpha$, and therefore $e_i$ has $i$ poles at $\alpha$ for all $2 \leq i \leq t-1$. By possibly translating $e_t$ with $e_{t-1}$, we find that $e_{t}$ also has at least $\gamma$ poles at $\alpha$. Since it already has $\gamma + 1$ poles at infinity, these $2\gamma + 1$ have to be all the poles of $e_{t}$. By Lemma \ref{Ldivisorgrowth} we find that $v_{\alpha}(e_{t+1}) > - \gamma$, we will show that this contradicts the fact that $e_{t+1} \in S_{t}^2$. 
	
	Define $T := \{(j-1, j) \mid 2 \leq j \leq t\}\cup\{(j, j), (j, t) \mid 1  \leq j \leq t\}\cup\{(1, 3)\}$ then, $\{e_ie_j \mid (i, j) \in T\}$ is a basis of $S_{t}^2$, these are marked as bold numbers in Table \ref{TABgamma1III}. Therefore there are $c_{i, j} \in K$ such that 
	$$e_{t+1} = \sum_{(i, j) \in T} c_{i, j}e_{i}e_j.$$ 
	Consider the elements $e_ie_j$ of maximal degree for which $c_{i, j} \neq 0$, there are either one or two of such elements. There is one element if and only if $\deg(e_ie_j)= t+1 = \gamma +2 $, which is only possible if $i, j \neq t$. Hence, $v_{\alpha}(e_ie_j) = -(\gamma + 2)$ and all other elements for which $c_{i, j} \neq 0$, will have even less poles at $\alpha$. By the ultrametric property we find, $v_\alpha(e_{t+1}) = - (\gamma + 2)$, a contradiction. In case that there are two elements of maximal degree with $c_{i, j}, c_{i', j'} \neq 0$, we know that this degree must be strictly bigger than $t+1$. One of them will be of the form $j = t$ and thus have $\gamma + i$ poles at $\alpha$, the other will have $i', j' \neq t$ and therefore the number of poles at $\alpha$ will equal its degree $i'+j' = i + t + 1 = i + \gamma + 2$. By the ultrametric property we find that $v_{\alpha}(e_{t+1}) = -(i' + j') < -(\gamma + 2)$. We conclude that it is impossible for $e_{t+1}$ to have at most $\gamma$ poles at $\alpha$, proving that no $S$ exists. \end{proof}

	\begin{table} 
	\begin{center}
		\begin{tabular}{c | c | c | c | c | c | c | c | c | c}
			&$e_1$&$e_2$&$e_3$& $e_4$ &$\ldots$ &$e_{t-2}$& $e_{t-1}$ & $e_{t}$ & $e_{t+1}$ \\\hline
			$e_1$& $\bm0$ & $\bm2$& $\bm3$ & $4$ & $\ldots$ &$t-2$ & $t-1$ & $\bm{t}$ & $t+1$ \\\hline
			$e_2$& - & $\bm4$ & $\bm5$& $6$ & $\ldots$ &$t$ & $t+1$ & $\bm{t+2}$ & $t+3$\\\hline
			$e_3$& - & - & $\bm6$& $\bm7$& $\ldots$ &$t+1$ & $t+2$ &$ \bm {t+3}$ & $t+ 4$\\\hline
			$e_4$& - & - & -& $\bm8$& $\ldots$ &$t+2$ & $t+3$ & $\bm{t+4}$ & $t+ 5$\\\hline
			$\vdots$& $\vdots$ & $\vdots$ & $\vdots$ & $\vdots$& $\bm\ddots$ & $\vdots$ & $\vdots$ & $\bm\vdots$ & $\vdots$ \\\hline
			$e_{t-2} $& - & - & -&-& $\ldots$ &$\bm{2t-4}$ & $\bm{2t-3}$ & $\bm{2t-2}$ & $2t-1$  \\\hline
			$e_{t-1}$ & - & - & -&-&  $\ldots$ &- & $\bm{2t-2}$ & $\bm{2t-1}$& $2t$\\\hline
			$e_{t}$ &-&-&-&-&$\ldots$&-&- & $\bm{2t}$ & $2t+1$\\\hline
			$e_{t+1}$ & - & - &-& - & $\ldots$&-&-&-&$2t+2$ \\
		\end{tabular}
		\caption{The degree table of $S_{t+1}^2$ when $t_1 = 3$. In each cell the degree of the product $e_{i}e_j$ is written. The bold numbers correspond to a possible basis for $S_t^2$.}\label{TABgamma1III}
	\end{center}
\end{table}

%	\begin{table} 
%	\begin{center}
%		\begin{tabular}{c | c | c | c | c | c | c | c | c | c}
%			&$e_1$&$e_2$&$e_3$& $e_4$ &$\ldots$ &$e_{t-2}$& $e_{t-1}$ & $e_{t}$ & $e_{t+1}$ \\\hline
%			$e_1$& $\bm0$ & $\bm2$& $\bm3$ & $4$ & $\ldots$ &$t-2$ & $t-1$ & $\bmt-1+\Delta$ & $t-1+2\Delta$ \\\hline
%			$e_2$& - & $\bm4$ & $\bm5$& $6$ & $\ldots$ &$t$ & $t+1$ & $\bmt+1+\Delta$ & $t+1+2\Delta$\\\hline
%			$e_3$& - & - & $\bm6$& $\bm7$& $\ldots$ &$t+1$ & $t+2$ &\color{red} $t+2+\Delta$ & $t+ 2+ 2\Delta$\\\hline
%			$e_4$& - & - & -& $\bm8$& $\ldots$ &$t+2$ & $t+3$ & $\bmt+3+\Delta$ & $t+ 3+ 2\Delta$\\\hline
%			$\vdots$& $\vdots$ & $\vdots$ & $\vdots$ & $\vdots$& $\bm\ddots$ & $\vdots$ & $\vdots$ & $\bm\vdots$ & $\vdots$ \\\hline
%			$e_{t-2} $& - & - & -&-& $\ldots$ &$\bm2t-4$ & $\bm2t-3$ & $\bm2t-3+\Delta$ & $2t-3+2\Delta$  \\\hline
%			$e_{t-1}$ & - & - & -&-&  $\ldots$ &- & $\bm2t-2$ & $\bm2t-2+\Delta$& $2t-2+2\Delta$\\\hline
%			$e_{t}$ &-&-&-&-&$\ldots$&-&- & $\bm2t-2+2\Delta$ & $2t-2+3\Delta$\\\hline
%			$e_{t+1}$ & - & - &-& - & $\ldots$&-&-&-&$2t-2 + 4\Delta$ \\
%		\end{tabular}
%		\caption{The degree table of $S_{t+1}^2$ for $t_1 = 3$. In each cell the degree of the product $e_{i}e_j$ is written. The red numbers correspond to a possible basis for $S_t^2$.}\label{TABgamma1III}
%	\end{center}
%\end{table}

	\subsubsection*{IV. The case $t=4$}
When $t=4$ the dimension of $S_t^2$ is at most $10$ and therefore $\gamma = 2$ or $\gamma = 3$. 

\begin{tl}\label{Lpossdeg}~
	\begin{enumerate}[\normalfont(i)]
		\item 	When $\gamma = 2$ the set of degrees of $S_5$ equals $\{0, 1, 2, 3, 4\}$, $\{0, 1, 3, 4, 5\}$, $\{0, 2, 3, 4, 5\}$ or $\{0, 3, 4, 5, 6\}$. 
		\item 	When $\gamma = 3$ the set of degrees of $S_6$ equals $\{0, 1, 2, 3, 4, 5\}, \{0, 2, 3, 4, 5, 6\},$ \\ $ \{0, 3, 4, 5, 6, 7\}, \{0, 4, 5, 6, 7, 8\}, \{0, 1, 3, 4, 5, 6\}$ or $\{0, 2, 4, 5, 6, 7\}$
	\end{enumerate}
\end{tl}
\begin{proof}
	We only give the proof for $\gamma = 2$, the case $\gamma = 3$ is very similar. Define $d := \deg(e_2), k := \deg(e_3) - \deg(e_2)$ and $\ell := \deg(e_4) - \deg(e_3)$, by Lemma \ref{Ldivisorgrowth}(i) we know that $\deg(e_5) = d + k + 2\ell$. Freiman's $3k-4$ Theorem then implies that $d+k+2\ell \leq 6$, which gives us finitely many options for $(d, k, \ell)$. For each of these options one may look at the degree table of $S_5^2$, if there are at least $12$ distinct values in this table this would imply $\gamma > 2$, which leaves us with the four cases mentioned in the lemma.  \qedhere

\end{proof}

We will give the explicit calculation for the case $\{0, 1, 3, 4, 5\}$ and then only state the results for the other cases. The degree table of $S_5^2$ is shown in Table \ref{TABdegfifth}, where the bold numbers point out a basis of $S_5^2$ and the non-bold ones are dependent on them. 

\begin{table} [h]
	\begin{center}
		\begin{tabular}{c | c | c | c | c | c}
			&$e_1$&$e_2$&$e_3$&$e_4$&$e_5$\\\hline
			$e_1$& $\bm0$ & $\bm1$& $\bm3$ & \color{black}{$4$}& \color{black}$5$\\\hline
			$e_2$& - & $\bm2$ & $\bm4$& $\bm5$& \color{black}$6$ \\\hline
			$e_3$& - & - & $\bm6$& $\bm7$& \color{black}$8$ \\\hline
			$e_4$& - & - & -& $\bm8$& $\bm9$ \\\hline
			$e_5$& - & - & -& -& $\bm{10}$ \\
		\end{tabular}
		\caption{The degree table of $S_5^2$ when the degrees in $S_5$ equal $0, 1, 3, 4, 5$.}\label{TABdegfifth}
	\end{center}
\end{table}

From this we can find that (after possibly translating the basis elements by multiples of each other) there must exist $a_1, a_2, \ldots, a_8 \in K$ such that the following relations hold: 
\begin{align*}
	&e_4 = e_2e_3 + a_1e_2^2, \\
	&e_5 = e_2e_4 + a_2e_2^2,\\
	&e_2e_5 = e_3^2 + a_3e_2e_4 + a_4e_2e_3 + a_5e_2^2 + a_6e_3 + a_7e_2  + a_8.
\end{align*}
By substituting $e_4$ and $e_5$ in the last equation we find a relation between $e_2$ and $e_3$ of degree at most $4$ in $e_2$ and degree at most $2$ in $e_3$. Recall that $F = K(x) = K(e_2. e_3)$, we can now conclude that $[F: K(e_2)] \leq 2$ and $[F : K(e_3)] \leq 4$. This implies that $e_2$ has at most $2$ poles and $e_3$ at most $4$, hence they both have at most $1$ pole outside of infinity. 

Assume that $e_3$ has a pole outside of infinity at $\alpha$ and $e_2$ does not. Then, by our first two relations, both $e_4$ and $e_5$ have exactly one pole at $\alpha$. However, our third relation would then imply that $e_2e_5$ has two poles at $\alpha$, which is a contradiction. We conclude, if $e_3$ has a pole outside of infinity, $e_2$ should have the same pole.

If $e_2$ does not have any poles outside of infinity then neither do $e_4$ and $e_5$, because $e_4 \in S_3^2$ and $e_5 \in S_4^2$. Here we find $S_5 \subset L(5P_\infty)$.

If $e_2$ does have a pole at $\alpha \neq \infty$, then $e_4$ has either $0, 1$ or $2$ poles there. However, if $e_4$ would have $2$ poles at $\alpha$, then $e_5$ would have $3$ and thus $e_2e_5$ has $4$, this cannot be true by the third relation, since the right hand side can have at most $3$. Here we find $S_5  \subset L(5P_\infty + P_\alpha)$. 

In both situations the basis can be continued. Using Lemma \ref{Ldivisorgrowth} we find that either $S \subset L(nP_\infty)$ or $S \subset L(nP_\infty + P_\alpha)$. 

\Change{
\begin{tl} Let $\gamma = 2$,
	\begin{enumerate}[\normalfont(i)]
		\item when $\deg(S_5) = \{0, 1, 2, 3, 4\}$ we find $S \subset L((n-1)P_\infty + P_\alpha)$ for $\alpha \in K$ or $S \subset L((n-1)P_\infty + P_{\alpha_1} + P_{\alpha_2})$ where $\alpha_1, \alpha_2 \in K$ might be equal.
		\item when $\deg(S_5) = \{0, 2, 3, 4, 5\}$ we find $S \subset L(nP_\infty)$ or $S \subset L(nP_\infty + P_\alpha)$ for some $\alpha \in K$. 
		\item when $\deg(S_5) = \{0, 1, 3, 4, 5\}$ we find $S \subset L(nP_\infty)$ or $S \subset L(nP_\infty + P_\alpha)$ for some $\alpha \in K$.
		\item when $\deg(S_5) = \{0, 3, 4, 5, 6\}$ we find $S \subset L((n+1)P_\infty)$. 
	\end{enumerate}
\end{tl}}

\begin{tl}\label{Le4}
	When $\gamma = 3$ we have $F = K(e_2, e_3)$.
\end{tl}
\begin{proof}
	We already know that $F = K(e_2, e_2, e_4)$, so it suffices to show that $e_4 \in K(e_2, e_3)$. By looking at the degree table of $S_5^2$ for any of the possible sets in Lemma \ref{Lpossdeg}(ii) we find that $e_5, e_2e_5 \in S_3S_4$ and $e_2e_5 \not\in S_3^2$. Therefore we have relations of the following form
	\begin{align*}
		e_5 &= \sum_{\substack{1 \leq i \leq 3 \\ i \leq j \leq 4}} a_{i, j}e_ie_j, \\
		e_2e_5 &= \sum_{\substack{1 \leq i \leq 3 \\ i \leq j \leq 4}} b_{i, j}e_ie_j,
	\end{align*}
for some $a_{i, j}, b_{i, j} \in K$. Substituting the first equation in the second we find a relation between $e_2, e_3$ and $e_4$, that is linear in $e_4$, showing that indeed $e_4 \in K(e_2, e_3)$. 
\end{proof}
\Change{
\begin{tl} Let $\gamma = 3$,
	\begin{enumerate}[\normalfont(i)]
		\item when $ \deg(S_6) = \{0, 4, 5, 6, 7, 8\}$ we find $S \subset L((n+2)P_\infty)$. 
		\item when $ \deg(S_6) = \{0, 3, 4, 5, 6, 7\}$ and $\{0, 2, 4, 5, 6, 7\}$ we find $S \subset L((n+1)P_\infty)$ or $S \subset L((n+1)P_\infty + P_\alpha)$ for some $\alpha \in K$. 
		\item when $ \deg(S_6) = \{0, 1, 3, 4, 5, 6\}$ and $\{0, 2, 3, 4, 5, 6\}$ we find $S \subset L(nP_\infty + P_\alpha)$ for some $\alpha \in K$ or $S \subset L(nP_\infty + P_{\alpha_1} + P_{\alpha_2})$ where $\alpha_1, \alpha_2 \in K$ might be equal. 
		\item when $ \deg(S_6) = \{0, 1, 2, 3, 4, 5\}$ we find $S \subset L((n-1)P_\infty + P_{\alpha_1} + P_{\alpha_2})$ or $S \subset L((n-1)P_\infty + P_{\alpha_1} + P_{\alpha_2} + P_{\alpha_3})$, where $\alpha_1, \alpha_2, \alpha_3 \in K$ might be equal.
	\end{enumerate}
\end{tl}
}

\Change{\begin{proof}We can prove (i)-(iii) using Lemma \ref{Le4} and the same ideas as shown above for $\gamma = 2$. In the following we prove the case $\deg(S_6) = \{0, 1, 2, 3, 4, 5\}$.
	
In Table \ref{TABdegsixths} the degree table of $S_6^2$ is shown. The bold numbers in the table represent a possible basis of $S_6^2$ in $E_6$. We find that $e_5$ equals $a e_3^2 + (1-a)e_2e_4 + \sum_{\substack{1 \leq i, j \leq 4 \\ 2 \leq i + j \leq 5}} a_{i, j} e_ie_j$ for some $a, a_{i, j} \in K$. We will distinguish between the cases $a \neq 0$ and $a = 0$. 

\begin{table} [h]
	\begin{center}
		\begin{tabular}{c | c | c | c | c | c | c}
						&$e_1$&$e_2$&$e_3$&$e_4$&$e_5$ & $e_6$ \\\hline
			$e_1$& $\bm0$ & $\bm1$& $\bm2$ & $\bm 3$& $4$ & $5$\\\hline
			$e_2$& - & $\bm2$ & $\bm 3$& $\bm4$& $5$ & $6$ \\\hline
			$e_3$& - & - & $\bm4$& $\bm5$& $6$ & $7$\\\hline
			$e_4$& - & - & -& $\bm6$& $\bm7$ & $8$\\\hline
			$e_5$& - & - & -& -& $\bm8$ & $\bm 9$ \\ \hline
			$e_6$& - & - & - & - & - &  $\bm {10}$
		\end{tabular}
		\caption{\Change{The degree table of $S_6^2$ when the degrees in $S_6$ equal $0, 1, 2, 3, 4, 5$. The bold numbers represent a possible basis of $S_6^2$.}}\label{TABdegsixths}
	\end{center}
\end{table}

In both cases, by analyzing the relations between the $e_i$, we find that either the lemma holds and $n+1 \leq \deg(D) \leq n + 2$ or the divisors $D_i$ take one of the following two forms:
\begin{alignat*}{3}
	D_1 &= 0 ~~~~~~~~&&D_i = (i-1)P_\infty + (i-1)P_\alpha &&\textup{ for all } 2 \leq i \leq n, \\
	D_1 &= 0 &&D_i=(i-1)P_\infty + (i-1)P_\alpha + P_\beta ~~~&&\textup{ for all } 2 \leq i \leq n.
\end{alignat*}
We will prove by contradiction that these two sets of divisors cannot occur. 

First assume $a \neq 0$, then we can replace $e_3^2$ in the basis by $e_5$. Since the dimension of $S' := \langle S_2S_4, e_5 \rangle$ is $8$ and the maximal degree occurring is $4$, we know that there must exist some element $C \in S'$ such that $\deg(C) \leq -3$. This also means that $C$ has $3$ poles, of which at most $2$ can be at $\beta$ and therefore $-v_\alpha(C) \geq 1$. We can write $C$ in the basis of $S'$ and find
$$C = c_{1, 1} + c_{1, 2}e_2 + c_{1, 3} e_3 + c_{1, 4} e_4 + c_{1, 5} e_5 + c_{2, 2}e_2^2 + c_{2, 3}e_2e_3 + c_{2, 4}e_2e_4,$$
for some $c_{i, j} \in K$. Furthermore, by possibly translating $e_6$ and $e_5$ with multiples of other basis elements, we find $e_6 = e_2e_5$. Therefore,
$$Ce_5 =  c_{1, 1}e_5 + c_{1, 2}e_2e_5 + c_{1, 3} e_3e_5 + c_{1, 4} e_4e_5 + c_{1, 5} e_5^2 + c_{2, 2}e_2e_6 + c_{2, 3}e_3e_6 + c_{2, 4}e_4e_6.$$
Hence, we have found an element in $S_6^2$ of degree at most $1$ with at least $5$ poles at $\alpha$. This is impossible. 

Next, assume $a = 0$ and $e_3^2 \in S_2S_6$. We can replace $e_3^2$ and $e_3e_4$ in the basis of $S_6^2$ by $e_6$ and $e_2e_5$ and we can choose a basis of $S$ such that $e_5 = e_2e_4$. The space $S'' := \langle S_2S_3, e_4 \rangle$ has dimension $6$ and maximal degree $3$. Thus, $S''$ contains an element $B$ such that $\deg(B) \leq -2$ and $-v_\alpha(B) \geq 0$. We can write $Be_4$ as follows
$$Be_4 = b_{1, 1}e_4 + b_{1, 2}e_2e_4 + b_{1, 3}e_3e_4 + b_{1, 4}e_4^2 + b_{2, 2}e_2e_5 + b_{2, 3}e_3e_5,$$
with $b_{i, j} \in K$. Using the relations between the $e_ie_j$ we can also find $b'_{i, j} \in K$ such that 
$$Be_4 = b'_{1, 4}e_4 + b'_{2, 4}e_2e_4 + b'_{1, 6}e_6 + b'_{2, 6}e_2e_6 + b'_{2, 5}e_2e_5 + b'_{2, 6}e_2e_6.$$
Multiplying this equation by $e_4$ again we find
$$Be_4^2 = b'_{1, 4}e_4^2 + b'_{2, 4}e_4e_5 + b'_{1, 6}e_4e_6 + b'_{2, 6}e_5e_6 + b'_{2, 5}e_5^2 + b'_{2, 6}e_5e_6.$$
Hence, we have found an element in $S_6^2$ of degree at most $4$ with at least $6$ poles at $\alpha$, which is impossible. 

Finally, assume $a = 0$ and $e_3^2 \not\in S_2S_6$. Using a relation for $e_5$ and $e_2e_5$ we can find a relation $fe_3 = g$ with $f, g \in K[e_2, e_4]$. Also looking at relations for $e_6$ and $e_2e_6$ we find $[F : K(e_4)] \leq 5$. This is a contradiction since $e_4$ has at least $6$ poles according to our current divisors. 

We conclude that it must always hold that $\deg(D) \in \{ n+1, n+2 \}$, confirming the lemma for this case. \end{proof}}

\section{Proof of Theorem \ref{Tmaxdiv}} \label{Supper}
Define $t$ to be such that $\gamma_t > \gamma_{t-1} = 0$, then by the assumption we know that \Change{$t > \Delta_{\textup{Max}}$ and $t > 3$}. Furthermore, by Theorem \ref{Tknownbasis}, we find that $e_1 = 1, e_2 = x, \ldots, e_{t-1} = x^{t-2}$ for some $x \in F$ such that $F = K(x)$. To give the proof we will need something slightly stronger than a filtered  basis with respect to $v_\infty$, which exists by the following lemma. 

\begin{tl}\label{Lsuperfilt}
	Let $K$ be an algebraically closed field and $S \subset K(x)$ an $n$-dimensional $K$-vector field such that $\min(\deg(S)) = 0$. Then it is possible to give a basis $\{e_i \mid 1 \leq i \leq n\}$ for $S$ that is filtered at $v_\infty$ such that for any $\alpha \in K$ with $m_\alpha := \min(v_\alpha(S)) < 0$ we have
	$$v_{\alpha}(e_1) \geq v_\alpha(e_2) \geq \ldots \geq v_{\alpha}(e_n) = m_\alpha.$$ 
\end{tl}
\begin{proof}
	Define $A := \{\alpha \in K \mid m_\alpha < 0\}$, then $A$ must be finite because the elements in the basis can only have a finite number of poles. There is some $i \geq 2$ such that for every $\alpha \in A$ we have $v_\alpha(e_1) \geq \ldots \geq v_\alpha(e_{i-1})$. We may replace $e_i$ by $e_i + a e_{i-1}$ with $a \in K$ and we will still have a filtered basis by the ultrametric property. For every $\alpha$ there is at most one $a_\alpha \in K$ such that $e_i + a_{\alpha}e_{i-1}$ has higher valuation in $v_\alpha$ than $e_{i-1}$ has. Take $a \in K$ such that $a \neq a_\alpha$ for all $\alpha \in A$ and replace $e_i$ by $e_i + ae_{i-1}$. Then we have $v_{\alpha}(e_1) \geq \ldots \geq v_{\alpha}(e_i)$ for all $\alpha \in A$. Repeating this process until $i = n$ gives a basis that meets the properties of the lemma.
\end{proof}

\begin{td}
	We call a basis meeting the criteria of Lemma \ref{Lsuperfilt} a \textit{super filtered basis} of $S$ with respect to $v_\infty$.
\end{td}

From now on we assume $\{e_1, \ldots, e_n\}$ to be a super filtered basis of $S$. We want to show that the dimension of $S_i^2$ grows roughly like the number of poles that $e_i$ has. Define $M_i$ as $\sum_{\alpha \in K} \max(-v_\alpha(e_i), 0)$, so it equals the number of poles with multiplicity that $e_i$ has outside of infinity. Note that because the basis is super filtered, we have $M_1 \leq M_2 \leq \ldots \leq M_n$ and $M_i + \deg(e_i) = \deg (D_i)$. Also let $\Delta_i := \deg(e_i) - \deg(e_{i-1})$ and $\mu_i := M_i - M_{i-1}$ for all $2 \leq i \leq n$, then we know that $t > \Delta_i + \mu_i$ for all such $i$. 
%We will prove by induction that 
%\begin{equation}\label{EQind}
%1 + M_i + \deg(e_i) \leq i + \gamma \textup{ for all } 1 \leq i \leq n.
%\end{equation}
%This shows that $\dim(L(D_i)) \leq i + \gamma$ and thus in particular $\dim(L(D)) \leq n + \gamma$. 
% Therefore, (\ref{EQind}) holds for all $i \leq t-1 $. Now we assume that (\ref{EQind}) holds for some integer $i \geq t-1$ and we will prove it also holds for $i + 1$.

\begin{tl} \label{Ltinssq}
	Let $\gamma_{t-1} = 0$ and suppose that $t > \mu_j + \Delta_j$ for all $1 \leq j \leq n$. Then, for any $t - 1 \leq i \leq n$ we have $T_{i} := L((t-2)P_\infty + D_{i}) \subset S_{i}^2$.
\end{tl}
\begin{proof}
	We will show this using induction. Note that $S_{t-1}^2 = \mathfrak p_{2t-4} = T_{t-1}$ so it holds for the base case. We assume $i \geq t$ and $T_{i-1} \subset S_{i-1}^2$ and we will now prove $T_i \subset S_i^2$. 
	
	Consider the elements $\{e_1e_i, \ldots, e_{\mu_i}e_i\} = \{e_i, xe_i, \ldots, x^{\mu_i-1}e_i\}$, these elements all have the same $M_i$ poles and degree at most $$\mu_i-1 + \deg(e_i)  \leq t - 2 + \deg(e_{i-1}).$$
	%Hier gebruik je dat $\mu_i + \Delta_i < t $
	By taking linear combinations of these elements we may create $\mu_i$ distinct elements $s_1, \ldots, s_{\mu_i}$ with $M_{i-1}+1, \ldots, M_{i-1} + \mu_i - 1$ and $M_i$ poles, respectively. Furthermore, we may choose exactly which poles these elements lose in comparison to $e_i$ and can therefore assure that $s_1$ has all poles that $e_{i-1}$ has and one extra and that $s_j$ for each $j \geq 2$ has the poles $s_{j-1}$ has and one extra. From this we may conclude that
	$$\langle T_{i-1}, s_1, \ldots, s_{\mu_i} \rangle = L((t-2)P_\infty + D_i - \Delta_iP_\infty) \subset S_i^2.$$ 
	Next we consider the elements $\{e_{t-\Delta_i} e_i, \ldots, e_{t-1}e_i \}$ which have degrees $t - 1 + \deg(e_{i-1}), t + \deg(e_{i-1}), \ldots, t - 2 + \deg(e_i)$, respectively. Hence, we find that all these elements are needed for a basis of $S_i^2$ compared to $\langle T_{i-1}, s_1, \ldots, s_{\mu_i} \rangle.$  Noting furthermore that their poles outside of infinity are exactly the $M_i$ poles of $e_i$ outside of infinity, we may conclude that $T_i \subset S_i^2$, which by the principal of induction proves the lemma. 
\end{proof}
\begin{proof}[Proof of Theorem \ref{Tmaxdiv}]
Using Lemma \ref{Ltinssq} we have that $T_n = L((t-2)P_\infty + D) \subset S^2$. Now we consider the elements $\{e_{t}e_n, \ldots, e_n^2\}$. They all have distinct degrees that are at least $t-2 + \deg(e_n) + 1$ and are thus needed for a basis of $S^2$ compared to $T_n$. Hence we find that 
%$$\langle T_n, e_te_n, \ldots, e_n^2 \rangle \subset S^2 ,$$  and 
$$\dim S^2 \geq \dim T_n + n - t + 1 = \deg(D) + n.$$ 
Rearranging everything we get $$\dim(D) \leq  \dim(S^2) - n + 1 = n + \gamma,$$ 
which proves the conjecture for this specific case. 
\end{proof}
\printbibliography
\emph{Mathematical Institute, Bunsenstrasse 3-5, 37073 G\"ottingen, Germany}\\
~~~ mieke.wessel@mathematik.uni-goettingen.de

\end{document}